\pdfoutput=1
\documentclass[a4paper,abstracton,DIV12]{scrartcl}
\usepackage[latin1]{inputenc}
\usepackage{graphicx} 
\usepackage{caption}
\usepackage{subcaption}
\usepackage{url}
\usepackage{amsmath}
\usepackage{amsfonts}
\usepackage{amssymb}
\usepackage{amsthm}
\usepackage{mathrsfs}
\usepackage{booktabs}
\usepackage{dcolumn}
\usepackage{enumitem}
\usepackage{bbold}
\newcommand{\OO}{\mathcal{O}}

\DeclareMathOperator{\dd}{d}
\DeclareMathOperator{\e}{e}

\newcommand{\Laplace}{\Delta}

\providecommand{\norm}[1]{\left\lVert#1\right\rVert}
\theoremstyle{plain}
\newtheorem{problem}{Problem}[section]
\newtheorem{theorem}[problem]{Theorem}

\newtheorem{corollary}[problem]{Corollary}

\newtheorem{assumption}[problem]{Assumption}

\newtheorem{remark}[problem]{Remark}

\newcommand{\optcrtl}{optimal control}

\title{Implicit Runge-Kutta schemes for optimal control problems with evolution equations}
\author{Thomas G. Flaig\thanks{thomas.flaig@unibw.de, Universit\"at der Bundeswehr M\"unchen, Institut f\"ur Mathematik und Bauinformatik, D-85577 Neubiberg, Germany}
} 
\date{October 28, 2013}
\date{}

\begin{document}

\maketitle

\begin{abstract}
	In this paper we discuss the use of implicit Runge-Kutta schemes for the time discretization 
	of optimal control problems with evolution equations.
	The specialty of the considered discretizations is  that 
	the discretizations schemes for the state and adjoint state are chosen 
	such that discretization and optimization commute. 
	It is well known that for Runge-Kutta schemes with this property additional order conditions are necessary.
	We give sufficient conditions for which class of schemes these additional order condition are automatically fulfilled.
	The focus is especially on implicit Runge-Kutta schemes of
	Gauss, Radau IA, Radau IIA, Lobatto IIIA, Lobatto IIIB and Lobatto IIIC collocation type up to order six.
	Furthermore we also use a SDIRK (singly diagonally implicit Runge-Kutta) method to demonstrate, that for general implicit Runge-Kutta methods the additional order conditions are not automatically fulfilled.

	Numerical examples illustrate the predicted convergence rates.

	\ 



\end{abstract}
%
%
\section{Introduction}

The novelty of this contribution is the characterization for which implicit Runge-Kutta schemes for distributed parabolic optimal control problems discretization and optimization commute and the convergence order is preserved. This characterization is done in terms of simplifying  assumptions for the coefficients of the schemes.
The commutability is desired for the following reasons. 
For the approach discretize-then-optimize we can choose an appropriate approximations for the state and the adjoint equation 
but we might need to transfer discrete quantities from one discretization to the other discretization. 
This may result in an solution operator which is not symmetric and positive definite. 
On the other hand if we chose the other approach optimize-then-discretize we do not have this problem, but we also do not know if the discrete adjoint state is an appropriate  approximation of the continuous adjoint state. 
Therefore our goal is to use schemes which combine the advantages of both approaches.

In particular we discuss higher order time discretization with implicit Runge-Kutta schemes for the \optcrtl\ problem 
\begin{align}
	\left.
	\begin{aligned}
		\min \frac12 \norm{M^{1/2} \left( y(\cdot,T)-y_D\right)}_H^2 
		&+  \int\limits_0^T 
		\frac\nu2 \norm{M^{1/2}u}_H^2\ \dd t,
	\\
	My_{t} + A y &= Bu,
	\\
	My(0)&=Mv,
	\end{aligned}
	\right\}
	\label{eq:intro:parabol:ocp}
\end{align}
with the control $u$ and the state $y$. The Hilbert space $H$ is appropriately chosen, the desired state $y_D\in H$ and the initial condition $v\in H$ are given, and the operators $M$ and $B$ are regular. 
Further we assume that the operator $A$ is self-adjoint, and maps $A:V\to V^*$ with the Hilbert space $V\subseteq H$. 
In the case of second order parabolic equations we choose $H^1(\Omega)\subseteq V \subseteq H^1_0(\Omega)$, corresponding to the boundary conditions, and $H=L^2(\Omega)$.

Due to the papers of Becker, Meidner and Vexler \cite{BeckerMeidnerVexler2007} and Meidner and Vexler \cite{MeidnerVexler2008a,MeidnerVexler2011} it is well known that discretization and optimization commute for continuous and discontinuous Galerkin time discretizations. The work on discontinuous Galerkin schemes \cite{MeidnerVexler2008a} provides error estimates for time discretization of arbitrary order, whereas the continuous Galerkin case was limited to the Petrov-Galerkin Crank-Nicolson scheme \cite{MeidnerVexler2011}.

Lasaint and Raviart \cite{LasaintRaviart1974} have proven the equivalence of discontinuous Galerkin time discretization to special implicit Runge-Kutta schemes. But there are also time stepping schemes, for which the equivalence to Galerkin schemes is not clear and for which discretization and optimization commute. 
A second order time stepping Crank-Nicolson scheme, for which discretization and optimization commute and which is not a Galerkin scheme, is discussed, among other variants, in a paper by Apel and Flaig \cite{ApelFlaig2012}. Previous papers on Crank-Nicolson time discretizations, as R\"osch \cite{Roesch2004}, did not provide results on second order convergence.

For the time discretization of optimal control problems it is well known, 
that Runge-Kutta schemes which provide the commutation of discretization and optimization need to fulfill additional order conditions, see 
Hager  \cite{Hager1976,Hager2000} and Bonnans and Laurent-Varin \cite{BonnansLaurent-Varin2004,BonnansLaurent-Varin2006}.
In \cite{BonnansLaurent-Varin2004,BonnansLaurent-Varin2006} no numerical discretization schemes, which fulfill these conditions, were given  and in \cite{Hager2000} only numerical examples with explicit Runge-Kutta schemes  were presented. 
The analysis was extended to W-method by Lang and Verwer and the additional order conditions up to order three can be found in \cite{LangVerwer2013}. 
Herty, Pareschi and Steffensen \cite{HertyPareschiSteffensen2013,HertyPareschiSteffensen2013b} transfer the theory of Hager and  Bonnans and Laurent-Varin to implicit-explicit discretizations, where the stiff part of the differential equation is discretized with an implicit scheme and the non-stiff part with an explicit scheme. They give order conditions up to order three. 

In this contribution we focus on $A$-stable discretization schemes for the discretization of a parabolic equation and therefore on implicit Runge-Kutta schemes. 
For schemes up to order six we give simple criteria for the decision whether the additional order conditions are fulfilled. 
These criteria are given in terms of well known simplifying assumptions on the coefficients of  Runge-Kutta schemes. 
In particular we see that collocation Runge-Kutta schemes of Gauss, Radau~IA, Radau~IIA, Lobatto~IIIA, Lobatto~IIIB and Lobatto~IIIC type fulfill the additional order conditions. 
We also give an SDIRK scheme as example for which the additional conditions do not hold and the order reduction can be observed.

The outline of the paper is as follows.
In the next section we introduce the time discretization and in Section \ref{sec:conv} we analyze under which circumstances the additional order conditions are fulfilled. 
In Section~\ref{sec:Num:Ex} numerical examples  confirm the predicted orders of convergence.

%
%
\section{Time Discretizations}
\label{sec:DiscrScheme}
%
%
\subsection{Runge-Kutta schemes for the time discretization of optimal control problems}
It is well known \cite{Lions1971,Troeltzsch2005} that the first order optimality conditions for the optimal control problem \eqref{eq:intro:parabol:ocp} are given by
\begin{align}
	\left.
	\begin{aligned}
	M\bar y_t + A\bar y &=B\bar u,
	&
	&
	&
	M\bar p_t-A\bar p &=0,
	\\
	M\bar y(0)&=M\bar v,
	&
	&
	&
	M\bar p(T)&=M \left(y_D-\bar y(T)\right),
	\\
	&
	&
	M\bar u &= \frac1\nu M\bar p.
	\end{aligned}
	\right\}
\end{align}
Since the problem \eqref{eq:intro:parabol:ocp} is convex these necessary optimality conditions are also sufficient. 
As seen in \cite[Formula (6)]{BonnansLaurent-Varin2006} and \cite{Hager1976,Hager2000} 
for the $s$-stage Runge-Kutta discretization of the optimal control problem \eqref{eq:intro:parabol:ocp}
 given by
\begin{align}
	\left.
	\begin{aligned}
	M \bar y_{k+1}
	&=
	M \bar y_k + \tau_k \sum_{i=1}^s b_i 
	\left(
		M \bar u_{k;i}-A\bar y_{k;i}
	\right)
	,
	\\
	M\bar y_{k;i}
	&=M\bar y_k+\tau_k \sum_{j=1}^s a_{ij} 
	\left(
		M\bar  u_{k;j}-A\bar y_{k;j}
	\right)
	,
	\\
	M\bar p_{k+1}
	&=M\bar p_k - \tau_k \sum_{i=1}^s \hat b_i 
		A \bar p_{k;i}
	,
	\\
	M\bar p_{k;i}
	&=M\bar p_k-\tau_k \sum_{j=1}^s \hat a_{ij} 
		A\bar  p_{k;j}
	,
	\\
	M\bar u_{k;i}&=\frac1\nu M\bar p_{k;i}
	,
	\\
	M\bar y_0&=Mv
	,
	\\
	M\bar p_N&=
	M\left(y_D-\bar y_N\right),
	\end{aligned}
	\right\}
	\label{eq:timediscr:rk}
\end{align}
discretization and optimization commute if the two schemes for the state and the adjoint state fulfill the conditions
\begin{align}
	\left.
	\begin{aligned}
	\hat b_i&=b_i,
	\\
	\hat a_{ij}&=b_j-\frac{b_j}{b_i}a_{ji}.
	\end{aligned}
	\right\}
	\label{eq:timediscr:connection:schemes}
\end{align}
In the discretization \eqref{eq:timediscr:rk} we denote the discretization of the state and the adjoint state for $t=t_k$ by $\bar y_k$, $\bar p_k$,  the inner stages of the Runge-Kutta schemes by $\bar y_{k;i}$, $\bar p_{k;i}$ and the time step size by  $\tau_k$.

The conditions \eqref{eq:timediscr:connection:schemes} are also known as condition for symplecticity of partitioned Runge-Kutta schemes \cite[Theorem VI.4.6]{HairerLubichWanner2006}. 

For the Runge-Kutta discretization of optimal control problems it is known (see \cite{BonnansLaurent-Varin2004,BonnansLaurent-Varin2006,Hager2000}) that in addition to the usual order conditions additional order conditions are needed. 
These conditions were given in \cite[Table 1]{Hager2000} up to order four and in \cite[Table 2--6]{BonnansLaurent-Varin2006} up to order six. 
We repeat these order conditions up to order four in Table \ref{tabel:timediscr:orderconditions}, the conditions of order five in Table~\ref{tabel:timediscr:orderconditions-o5} and the conditions of order six  in the Tables~\ref{tabel:timediscr:orderconditions-o6}--\ref{tabel:timediscr:orderconditions-o6-3}.
\begin{table}
	\caption{The order conditions for Runge-Kutta discretization for the state equation and optimal control problems, see also
	\cite[Table 2--4]{BonnansLaurent-Varin2006}\cite[Table~1 and~2]{Hager2000}. All summations go from $1$ to the number of stages $s$.
	\label{tabel:timediscr:orderconditions}
	}
	\centering
	\begin{subtable}{0.5\textwidth}
	\centering
		\caption{Abbreviations}
		\begin{align*}
			\toprule
				c_i&=\sum a_{ij},
				&
				\displaystyle d_j&=\sum b_ia_{ij}.
				\\
	                \bottomrule
			\nonumber
		\end{align*}
	\end{subtable}
	\begin{subtable}{\textwidth}
	\caption{Order conditions for the state equation without control \label{table:orderconditions:irk:ode}}
	\centering
	\begin{align}
		\toprule
		\text{Order}
		&&
		\text{Conditions}
		\nonumber
		\\
		\midrule
		1&
		&
		\sum b_i&=1.
		\tag{O1}
		\label{eq:O-1}
		\\
		\midrule
		2
		&
		&
		\sum d_i&=\frac12.
		\tag{O2}
		\label{eq:O-2}
		\\
		\midrule
		3
		&
		&
		\sum c_id_i &= \frac16,
		&
		\sum b_i c_i^2&=\frac13.
		\tag{O3}
		\label{eq:O-3}
		\\
		\midrule
		4
		&
		&
		\sum b_i c_i^3 &= \frac14,
		&
		\sum b_i c_i a_{ij}c_j &=\frac18,
		&
		\sum d_ic_i^2&=\frac1{12},
		&
		\sum d_ia_{ij}c_j&=\frac1{24}.
		\tag{O4}
		\label{eq:O-4}
		\\
		\bottomrule
		\nonumber
	\end{align}
	\end{subtable}

	\begin{subtable}{\textwidth}
		\caption{Additional order conditions for optimal control problems\label{table:orderconditions:irk:oc}}
	\centering
	\begin{align}
		\toprule
		\text{Order}
		&&
		\text{Additional }&\text{conditions}
		\nonumber
		\\
		\midrule
		3
		&
		&
		\sum \frac{d_i^2}{b_i}&=\frac13.
		\tag{A3}
		\label{eq:A-O-3}
		\\
		\midrule
		4
		&
		&
		\sum c_i\frac{d_i^2}{b_i}&=\frac1{12},
		&
		\sum \frac{d_i^3}{b_i^2}&=\frac14,
		&
		\sum \frac{b_i}{b_j}c_i a_{ij}d_j &=\frac5{24},
		&
		\sum \frac{d_i}{b_j} a_{ij}d_j &=\frac18.
		\tag{A4}
		\label{eq:A-O-4}
		\\
		\bottomrule
		\nonumber
	\end{align}
	\end{subtable}

\end{table}
\begin{table}
	\caption{The order conditions of order 5 for Runge-Kutta discretization for the state equation and optimal control problems, see also
	\cite[Table 6]{BonnansLaurent-Varin2006}. All summations go from $1$ to the number of stages $s$.
	\label{tabel:timediscr:orderconditions-o5}
	}
	\begin{subtable}{\textwidth}
		\caption{Order conditions of order 5 for the state equation without control \label{table:orderconditions:irk:ode-o5} (computed with Mathematica).}
		\begin{align}
			\toprule
			\sum b_i a_{ik}a_{kj}c_ic_j&=\frac1{30},
			&
			\sum a_{jk}c_jd_jc_k&=\frac1{40},
			&
			\displaystyle \sum b_ia_{ij}c_ic_j^2&=\frac1{15},
			\tag{O5-1}
			\label{eq:O-5-1}
			\\
			\midrule
			\sum c_j^3d_j &= \frac1{20},
			&
			\sum \frac{b_i}{b_k} a_{lk}a_{il}c_id_k &= \frac{11}{120},
			&
			\sum b_ia_{ij}c_i^2c_j &= \frac1{10},
			\tag{O5-2}
			\label{eq:O-5-2}
			\\
			\midrule
			\sum a_{kj}c_j^2d_k &= \frac1{60},
			&
			\sum b_ic_i^4 &= \frac15,
			&
			\sum b_i\left(\sum a_{ij}c_j\right)^2 &= \frac1{20}
			\tag{O5-3}
			\label{eq:O-5-3}
			\\
			\bottomrule
			\nonumber
		\end{align}
	\end{subtable}
	
	\begin{subtable}{\textwidth}
		\caption{Additional order conditions of order 5 for optimal control problems (see also \cite[Table 6]{BonnansLaurent-Varin2006}).
		\label{table:orderconditions:irk:oc-o5}
		}
		\begin{align}
			\toprule
			\sum \frac{1}{b_k}a_{lk}c_kd_kd_l &= \frac1{40},
			&
			\sum \frac1{b_k}c_k^2d_k^2&=\frac1{30},
			&
			\displaystyle \sum \frac1{b_l^2}c_ld_l^3 &=\frac1{20},
			\tag{A5-1}
			\label{eq:A-O-5-1}
			\\
			\midrule
			\sum \frac1{b_k}a_{kl}d_k^2c_l &= \frac1{60},
			&
			\sum \frac1{b_m^3}d_m^4 &= \frac15,
			&
			\sum b_i a_{ik}a_{ij}c_jc_k &= \frac{1}{20},
			\tag{A5-2}
			\label{eq:A-O-5-2}
			\\
			\midrule
			\sum a_{lk} a_{kj}c_jd_l&=\frac1{120},
			&
			\sum \frac1{b_k}a_{lk}d_kc_ld_l&=\frac7{120},
			&
			\sum\frac{b_ib_j}{b_k}a_{jk}a_{ik}c_ic_j &= \frac2{15},
			\tag{A5-3}
			\label{eq:A-O-5-3}
			\\
			\midrule
			\sum \frac{b_i}{b_k}a_{ik}c_ic_kd_k &= \frac7{120},
			&
			\sum\frac{b_i}{b_l^2}a_{il}c_id_l^2 &= \frac3{20},
			&
			\displaystyle\sum\frac{1}{b_k}a_{mk}a_{lk}d_ld_m &= \frac1{20},
			\tag{A5-4}
			\label{eq:A-O-5-4}
			\\
			\midrule
			\sum \frac1{b_l^2}a_{ml}d_l^2d_m&=\frac1{10},
			&
			\sum\frac{1}{b_k}a_{ml}a_{lk}d_kd_m&=\frac1{30},
			&
			\sum \frac{b_i}{b_k}a_{lk}a_{ik}c_id_l &= \frac3{40},
			\tag{A5-5}
			\label{eq:A-O-5-5}
			\\
			\midrule
			\sum \frac{b_i}{b_k}a_{ik}a_{il}d_kc_l &= \frac3{40},
			&
			\sum \frac{b_i}{b_lb_m} a_{im}a_{il}d_ld_m &= \frac2{15},
			&
			\sum\frac{b_i}{b_k}a_{ik}c_i^2d_k &= \frac3{20},
			\tag{A5-6}
			\label{eq:A-O-5-6}
			\\
			\midrule
			&&
			\sum \frac1{b_lb_m}a_{lm}d_l^2d_m &= \frac1{15}.
			\tag{A5-7}
			\label{eq:A-O-5-7}
			\\
			\bottomrule
			\nonumber
		\end{align}
	\end{subtable}
\end{table}
\begin{table}
	\caption{The order conditions of order 6 for Runge-Kutta discretization for the state equation  without control \label{table:orderconditions:irk:ode-o6} (computed with Mathematica).
	All summations go from $1$ to the number of stages $s$.
	\label{tabel:timediscr:orderconditions-o6}
	}
		\begin{align}
			\toprule
			\sum c_j^4d_j &= \frac1{30},
			&
			\sum a_{lm}a_{kl}a_{jk}d_jc_m &= \frac1{720},
			&
			\sum b_i a_{ij}c_i^2c_j^2&=\frac1{18},
			\tag{O6-1}
			\label{eq:O-6-1}
			\\
			\midrule
			\sum a_{jk}c_jd_jc_k^2 &=\frac1{90},
			&
			\sum b_ia_{ij}c_ic_j^3&=\frac1{24},
			&
			\sum a_{kj}c_j^3d_k &=\frac1{120},
			\tag{O6-2}
			\label{eq:O-6-2}
			\\
			\midrule
			\sum b_i a_{ij}c_i^3c_j &= \frac1{12},
			&
			\sum a_{jk}c_j^2d_jc_k&=\frac1{60},
			&
			\sum b_i a_{ij}a_{jk}c_ic_jc_k &=\frac1{48},
			\tag{O6-3}
			\label{eq:O-6-3}
			\\
			\midrule
			\sum a_{lj} a_{jk}c_jc_kd_l &= \frac1{240}
			&
			\sum a_{kl}a_{jk}c_jd_jc_l
			&=\frac1{180}
			&
			\sum b_i a_{jk}a_{ij}c_i^2c_k &=\frac1{36}
			\tag{O6-4}
			\label{eq:O-6-4}
			\\
			\midrule
			\sum b_i a_{ik} a_{kj}c_ic_j^2 &=\frac1{72}
			&
			\sum a_{lk}a_{kj} c_j^2 d_l &=\frac1{360}
			&
			\sum b_i a_{ik}a_{ij}c_j^2 c_k &= \frac1{36}
			\tag{O6-5}
			\label{eq:O-6-5}
			\\
			\midrule
			\sum b_i a_{il}a_{ik}a_{kj}c_jc_l &=\frac1{72},
			&
			\sum b_i a_{kl}a_{jk}a_{ij}c_ic_l
			&=\frac1{144},
			&
			\sum b_ic_i^5 &=\frac16,
			\tag{O6-6}
			\label{eq:O-6-6}
			\\
			\midrule
			\sum b_ic_i \left(\sum a_{ij}c_j\right)^2
			&=\frac1{24},
			&
			\sum b_i a_{ij} \left( \sum a_{jk}c_k\right)^2 &= \frac1{120}.
			\tag{O6-7}
			\label{eq:O-6-7}
						\\
			\bottomrule
			\nonumber
		\end{align}
\end{table}
\begin{table}
	\caption{Part 1 of the additional order conditions of order 6 for Runge-Kutta discretization for optimal control problems, see also
	\cite[Table 6]{BonnansLaurent-Varin2006}. All summations go from $1$ to the number of stages $s$.
	\label{tabel:timediscr:orderconditions-o6-1}
	}
	\begin{align}
		\toprule
		\sum \frac1{b_n^4}d_n^5 &= \frac16,
		&
		\sum \frac1{b_m^3}c_md_m^4 &= \frac1{30},
		&
		\sum \frac1{b_l^2}c_l^2d_l^3 &= \frac1{60},
		\tag{A6-1}
		\label{eq:A-O-6-1}
		\\
		\midrule
		\sum \frac1{b_k}c_k^3d_k^2 &=\frac1{60},
		&
		\sum \frac{b_i}{b_k}a_{ik}c_i^2c_kd_k &=\frac2{45},
		&
		\sum \frac1{b_k} a_{lk}c_kd_kc_ld_l&=\frac1{72},
		\tag{A6-2}
		\label{eq:A-O-6-2}
		\\
		\midrule
		\sum \frac{b_i}{b_l^2}a_{il}c_i^2d_l^2 &= \frac{19}{180},
		&
		\sum\frac1{b_l^2}a_{ml}d_l^2c_md_m &=\frac2{45},
		&
		\sum\frac1{b_m^2b_n}a_{nm}d_m^2d_n^2&=\frac1{18},
		\tag{A6-3}
		\label{eq:A-O-6-3}
		\\
		\midrule
		\sum \frac1{b_k}a_{kl}d_k^2c_l^2 &= \frac1{180},
		&
		\sum\frac{a_{lm}d_l^2c_md_m}{b_lb_m} &=\frac1{90},
		&
		\sum \frac{b_i}{b_k}a_{ik}c_i^3d_k&=\frac7{60},
		\tag{A6-4}
		\label{eq:A-O-6-4}
		\\
		\midrule
		\sum \frac{b_i}{b_k}a_{ik}c_ic_k^2d_k &=\frac1{40},
		&
		\sum\frac1{b_k}a_{lk}c_k^2 d_kd_l &= \frac1{120},
		&
		\sum\frac1{b_k}a_{lk}d_kc_l^2d_l&=\frac1{30},
		\tag{A6-5}
		\label{eq:A-O-6-5}
		\\
		\midrule
		\sum \frac{b_i}{b_l^2}a_{il}c_ic_ld_l^2&=\frac1{30},
		&
		\sum\frac1{b_l^2}a_{ml}c_ld_l^2d_m &=\frac1{60},
		&
		\sum\frac1{b_k}a_{kl}c_kd_k^2c_l&=\frac1{120},
		\tag{A6-6}
		\label{eq:A-O-6-6}
		\\
		\midrule
		\sum\frac{a_{lm}c_ld_l^2d_m}{b_lb_m}&=\frac1{40},
		&
		\sum\frac{b_i}{b_m^3}a_{im}c_id_m^3 &= \frac7{60},
		&
		\sum\frac1{b_m^3}a_{nm}d_m^3d_n &= \frac1{12}.
		\tag{A6-7}
		\label{eq:A-O-6-7}
		\\
		\bottomrule
		\nonumber
	\end{align}
\end{table}
\begin{table}
	\caption{Part 2 of the additional order conditions of order 6 for Runge-Kutta discretization for optimal control problems, see also
	\cite[Table 6]{BonnansLaurent-Varin2006}. All summations go from $1$ to the number of stages $s$.
	\label{tabel:timediscr:orderconditions-o6-2}
	}
	\begin{align}
		\toprule
		\sum \frac{a_{lm}d_l^3c_m }{b_l^2} &= \frac1{120},
		&
		\sum\frac{a_{nl}a_{ml}d_ld_md_n}{b_l^2} &= \frac1{24},
		&
		\sum \frac{a_{mk}a_{lk}c_kd_ld_m }{b_k} &= \frac1{120}
		\tag{A6-8}
		\label{eq:A-O-6-8}
		\\
		\midrule
		\sum \frac{a_{mn}d_m^3d_n }{b_m^2b_n} &=\frac1{24},
		&
		\sum \frac{a_{ml}a_{lk}d_kc_ld_m }{b_k}&=\frac1{80},
		&
		\sum \frac{b_i a_{im} a_{il}c_id_ld_m}{b_lb_m} &= \frac{11}{120},
		\tag{A6-9}
		\label{eq:A-O-6-9}
		\\
		\midrule
		\sum\frac{a_{nl}a_{lm}d_ld_md_n}{b_lb_m} &= \frac1{48},
		&
		\sum\frac{a_{nm}a_{nl}d_ld_md_n}{b_lb_m} &=\frac1{24},
		&
		\sum \frac{b_ib_ja_{jk}a_{ik}c_ic_jc_k  }{b_k} &=\frac1{24},
		\tag{A6-10}
		\label{eq:A-O-6-10}
		\\
		\midrule
		\sum \frac{b_i b_j a_{jl}a_{il}c_ic_jd_l}{b_l^2} &=\frac{11}{120},
		&
		\sum \frac{b_i a_{lk}a_{il}c_id_kc_l}{b_k} &= \frac{11}{240},
		&
		\sum\frac{b_ia_{lk}a_{ik}c_ic_kd_l}{b_k} &= \frac1{60},
		\tag{A6-11}
		\label{eq:A-O-6-11}
		\\
		\midrule
		\sum \frac{b_i a_{ml}a_{il}c_id_ld_m}{b_l^2} &= \frac7{120},
		&
		\sum\frac{b_i a_{lm}a_{il}c_id_ld_m}{b_lb_m}&=\frac{11}{240},
		&
		\sum b_ia_{ik}a_{ij}c_ic_jc_k &=\frac1{24}
		\tag{A6-12}
		\label{eq:A-O-6-12}
		\\
		\midrule
		\sum \frac{b_i a_{ik}a_{il}c_id_kc_l}{b_k}&=\frac7{120},
		&
		\sum\frac{a_{mk}a_{kl}d_kc_ld_m}{b_k} &=\frac1{240},
		&
		\sum \frac{b_ia_{ik}a_{kl}c_id_kc_l}{b_k} &=\frac1{80},
		\tag{A6-13}
		\label{eq:A-O-6-13}
		\\
		\midrule
		\sum \frac{b_i a_{im}a_{lm}c_id_l^2}{b_lb_m}&=\frac7{180},
		&
		\sum a_{jl}a_{jk}d_jc_kc_l &= \frac1{120},
		&
		\sum\frac{a_{lk}a_{lm}d_kd_lc_m}{b_k}&=\frac1{60},
		\tag{A6-14}
		\label{eq:A-O-6-14}
		\\
		\midrule
		\sum{b_i a_{lk}a_{ik}c_i^2d_l}{b_k}&=\frac{19}{360},
		&
		\sum\frac{a_{mk}a_{lk}c_ld_ld_m}{b_k}&=\frac1{45},
		&
		\sum\frac{a_{nm}a_{lm}d_l^2d_n}{b_lb_m} &=\frac1{36},
		\tag{A6-15}
		\label{eq:A-O-6-15}
		\\
		\midrule
		\sum\frac{a_{lm}a_{kl}d_k^2c_m}{b_k} &= \frac1{360},
		&
		\sum \frac{b_ia_{im}a_{ml}c_id_l^2}{b_l^2} &= \frac{13}{180},
		&
		\sum\frac{a_{nm}a_{ln}d_l^2d_m}{b_lb_m}&=\frac1{72},
		\tag{A6-16}
		\label{eq:A-O-6-16}
		\\
		\midrule
		\sum\frac{a_{nm}a_{ml}d_l^2d_n}{b_l^2}&=\frac1{36},
		&
		\sum \frac{b_ia_{im}a_{il}d_l^2c_m}{b_l^2}&=\frac{19}{360},
		&
		\sum\frac{b_ia_{in}a_{im}d_m^2d_n}{b_m^2b_n} &= \frac7{72},
		\tag{A6-17}
		\label{eq:A-O-6-17}
		\\
		\midrule
		\sum \frac{b_i a_{ik}a_{lk}c_ic_ld_l}{b_k} &= \frac{13}{360},
		&
		\sum\frac{a_{mk}a_{lm}d_kc_ld_l}{b_k}&=\frac7{360},
		&
		\sum \frac{b_i a_{il}a_{lk}c_ic_kd_k}{b_k}&=\frac7{360},
		\tag{A6-18}
		\label{eq:A-O-6-18}
		\\
		\midrule
		\sum\frac{a_{ml}a_{lk}c_kd_kd_m}{b_k}&=\frac1{180},
		&
		\sum\frac{b_ia_{il}a_{ik}c_kd_kc_l}{b_k} &=\frac1{45},
		&
		\sum \frac{b_i a_{im}a_{il}c_ld_ld_m}{b_lb_m} &= \frac{13}{360},
		\tag{A6-19}
		\label{eq:A-O-6-19}
		\\
		\midrule
		\sum\frac{b_i b_j a_{jk}a_{ik}c_i^2c_j}{b_k}&=\frac7{72},
		&
		\sum{b_ia_{lk}a_{il}c_i^2d_k}{b_k}&=\frac{13}{180},
		&
		\sum \frac{b_ia_{ik}a_{il}d_kc_l^2}{b_k}&=\frac7{180}.
		\tag{A6-20}
		\label{eq:A-O-6-20}
		\\
		\bottomrule
		\nonumber
	\end{align}
\end{table}
\begin{table}
	\caption{Part 3 of the additional order conditions of order 6 for Runge-Kutta discretization for optimal control problems, see also
	\cite[Table 6]{BonnansLaurent-Varin2006}. All summations go from $1$ to the number of stages $s$.
	\label{tabel:timediscr:orderconditions-o6-3}
	}
	\begin{align}
		\toprule
		\sum\frac{b_ib_j a_{jl}a_{lk}a_{ik}c_ic_j}{b_k}&= \frac1{18},
		&
		\sum\frac{b_ib_ja_{jl}a_{jk}a_{ik}c_ic_l}{b_k}&=\frac7{144},
		\tag{A6-21}
		\label{eq:A-O-6-21}
		\\
		\midrule
		\sum\frac{b_ib_j a_{jm}a_{im}c_jd_l}{b_lb_m}&=\frac{61}{720},
		&
		\sum \frac{b_i a_{lm}a_{il}a_{ik}d_kc_m}{b_k}&=\frac7{360},
		\tag{A6-22}
		\label{eq:A-O-6-22}
		\\
		\midrule
		\sum\frac{b_i a_{im}a_{ml}a_{lk}c_id_k}{b_k}&=\frac{19}{720},
		&
		\sum\frac{b_i a_{im}a_{il}a_{lk}d_kc_m}{b_k}&=\frac{13}{360},
		\tag{A6-23}
		\label{eq:A-O-6-23}
		\\
		\midrule
		\sum\frac{b_i a_{im}a_{in}a_{nl}d_ld_m}{b_lb_m}&=\frac1{18},
		&
		\sum\frac{b_i a_{ik}a_{mk}a_{lm}c_id_l}{b_k}&=\frac7{360},
		\tag{A6-24}
		\label{eq:A-O-6-24}
		\\
		\midrule
		\sum\frac{a_{nk}a_{mn}a_{lm}d_kd_l}{b_k}&=\frac1{144},
		&
		\sum\frac{b_i a_{ik}a_{mk}a_{lk}c_id_l}{b_k} &= \frac{13}{360},
		\tag{A6-25}
		\label{eq:A-O-6-25}
		\\
		\midrule
		\sum \frac{a_{nm}a_{mk}a_{lk}d_ld_n}{b_k}&=\frac1{72},
		&
		\sum \frac{b_i a_{im}a_{ik}a_{lk}d_lc_m}{b_k}&=\frac{19}{720},
		\tag{A6-26}
		\label{eq:A-O-6-26}
		\\
		\midrule
		\sum\frac{b_ia_{im}a_{il}a_{nl}d_md_n}{b_lb_m}&=\frac7{144}
		.
		\tag{A6-27}
		\label{eq:A-O-6-27}
		\\
		\bottomrule
		\nonumber
	\end{align}
\end{table}

%
%
\subsection{Implicit Runge-Kutta discretizations for optimal control problems}
For our discussion we focus on implicit collocation Runge-Kutta schemes of of Gauss, Radau~IA, Radau~IIA, Lobatto~IIIA, Lobatto~IIIB and Lobatto~IIIC type up to order 6 and a SDIRK method of order four.
The corresponding Butcher tableaux are repeated in Table~\ref{table:irk:2:3}--\ref{table:irk:6}.
In the selection of schemes the focus was on $A$-stable Runge-Kutta schemes of higher order. 
Additionally the St\"ormer Verlet scheme of order two was included, as this gives a new variant of the results of \cite{ApelFlaig2012,Flaig2013}. 
Whereas in  \cite{ApelFlaig2012,Flaig2013} the state and the adjoint state were discretized on shifted time meshes, in the discretization~\eqref{eq:timediscr:rk} the state and the adjoint state are discretized on the same time mesh.
The corresponding discretization schemes for the adjoint equation are given by the relation 
\eqref{eq:timediscr:connection:schemes}. 
\begin{table}
	\caption{Coefficients of Runge Kutta schemes of order two and three (see also ~\cite{HairerLubichWanner2006,HairerNorsettWanner1987,HairerWanner1991}).
	\label{table:irk:2:3}}
	\begin{subtable}{0.325\textwidth}
		\caption{Coefficients of the St\"ormer-Verlet discretization for the state (cf. \cite[Table 2.1]{HairerLubichWanner2006}).
		\label{table:irk:SV:2}
		}
		\centering
		\begin{align*}
			\begin{matrix}
				0	
				& \vline	
				&
				0 		
				& 
				0
				\\
				&\vline
				\\
				1	
				&\vline		
				&
				\frac12	
				&
				\frac12
				\\
				&\vline
				\\
				\hline
				&\vline
				\\
				&\vline	
				&
				\frac12	
				&
				\frac12
			\end{matrix}
		\end{align*}
	\end{subtable}
	\begin{subtable}{0.325\textwidth}
		\caption{Coefficients for the Radau~IA method of order three (cf.~\cite[Table IV.5.3.]{HairerWanner1991}).
			\label{table:irk:RadauIA:3}
		}
		\centering
		\begin{align*}
			\begin{matrix}
				0
				&
				\vline
				&
				\frac14
				&
				-\frac14
				\\
				&\vline
				\\
				\frac23
				&\vline
				&
				\frac14
				&
				\frac5{12}
				\\
				&\vline
				\\
				\hline
				&\vline
				\\
				&\vline
				&
				\frac14
				&
				\frac34
			\end{matrix}
		\end{align*}
	\end{subtable}
	\begin{subtable}{0.325\textwidth}
		\caption{Coefficients for the Radau~IIA method of order three (cf.~\cite[Table IV.5.5.]{HairerWanner1991}).
			\label{table:irk:RadauIIA:3}
		}
		\centering
		\begin{align*}
			\begin{matrix}
				\frac13
				&\vline
				&
				\frac5{12}
				&
				-\frac1{12}
				\\
				&\vline
				\\
				1
				&\vline
				&
				\frac34
				&
				\frac14
				\\
				&\vline
				\\
				\hline
				&\vline
				\\
				&\vline
				&
				\frac34
				&
				\frac14
			\end{matrix}
		\end{align*}
	\end{subtable}
\end{table}
\begin{table}
	\caption{Coefficients of Runge Kutta schemes of order four (see also ~\cite{HairerLubichWanner2006,HairerNorsettWanner1987,HairerWanner1991}).
	\label{table:irk:4}}

	\begin{subtable}{0.4\textwidth}
		\caption{Coefficients for the Gauss scheme of order four (cf.~\cite[Table II.7.3]{HairerNorsettWanner1987},\cite[Table IV.5.1.]{HairerWanner1991}).
			\label{table:irk:Gauss:4}
			}
		\centering
		\begin{align*}
			\begin{matrix}
				\frac12 -\frac{\sqrt{3}}6
				&\vline
				&
				\frac14 
				&
				\frac14-\frac{\sqrt{3}}6
				\\
				&\vline
				\\
				\frac12 +\frac{\sqrt{3}}6
				&
				\vline
				&
				\frac14+\frac{\sqrt{3}}6
				&
				\frac14
				\\
				&\vline
				\\
				\hline
				&\vline
				\\
				&\vline
				&
				\frac12
				&
				\frac12
			\end{matrix}
		\end{align*}
	\end{subtable}
		\begin{subtable}{0.5\textwidth}
		\caption{Coefficients for an $L$-stable SDIRK method of order four (cf.~\cite[Formula (6.16)]{HairerWanner1991}).
			\label{table:irk:SDIRK:4}
		}
		\centering
		\begin{align*}
			\begin{matrix}
				\frac14
				&\vline
				&
				\frac14
				\\
				&\vline
				\\
				\frac34
				&\vline
				&
				\frac12
				&
				\frac14
				\\
				&\vline
				\\
				\frac{11}{20}
				&\vline
				&
				\frac{17}{50}
				&
				-\frac{1}{25}
				&
				\frac14
				\\
				&\vline
				\\
				\frac12
				&\vline&
				\frac{371}{1360}
				&
				-\frac{137}{2720}
				&
				\frac{15}{544}
				&
				\frac14
				\\
				&\vline
				\\
				1
				&\vline
				&
				\frac{25}{24}
				&
				-\frac{49}{48}
				&
				\frac{125}{16}
				&
				-\frac{85}{12}
				&
				\frac14
				\\
				&\vline
				\\
				\hline
				&\vline
				\\
				&\vline
				&
				\frac{25}{24}
				&
				-\frac{49}{48}
				&
				\frac{125}{16}
				&
				-\frac{85}{12}
				&
				\frac14
			\end{matrix}
		\end{align*}
	\end{subtable}

	\begin{subtable}{0.325\textwidth}
		\caption{Coefficients for the Lobatto~IIIA method of order four (cf.~\cite[Table IV.5.7.]{HairerWanner1991}).
			\label{table:irk:LobattoIIIA:4}
		}
		\centering
		\begin{align*}
			\begin{matrix}
				0
				&
				\vline
				&
				0
				&
				0
				&
				0
				\\
				&\vline\\
				\frac12
				&
				\vline
				&
				\frac5{24}
				&
				\frac13
				&
				-\frac1{24}
				\\
				&\vline\\
				1
				&
				\vline
				&
				\frac16
				&
				\frac23
				&
				\frac16
				\\
				&\vline
				\\
				\hline
				&\vline
				\\
				&\vline
				&
				\frac16
				&
				\frac23
				&
				\frac16
			\end{matrix}
		\end{align*}
	\end{subtable}
	\begin{subtable}{0.325\textwidth}
		\caption{Coefficients for the Lobatto~IIIB method of order four (cf.~\cite[Table IV.5.9.]{HairerWanner1991}).
			\label{table:irk:LobattoIIIB:4}
		}
		\centering
		\begin{align*}
			\begin{matrix}
				0
				&
				\vline
				&
				\frac16
				&
				-\frac16
				&
				0
				\\
				&\vline
				\\
				\frac12
				&\vline
				&
				\frac16
				&
				\frac13
				&
				0
				\\
				&\vline
				\\
				1
				&\vline
				&
				\frac16
				&
				\frac56
				&
				0
				\\
				&\vline
				\\
				\hline
				&\vline
				\\
				&\vline
				&
				\frac16
				&
				\frac23
				&
				\frac16
			\end{matrix}
		\end{align*}
	\end{subtable}
	\begin{subtable}{0.325\textwidth}
		\caption{Coefficients for the Lobatto~IIIC method of order four (cf.~\cite[Table IV.5.11.]{HairerWanner1991}).
			\label{table:irk:LobattoIIIC:4}
		}
		\centering
		\begin{align*}
			\begin{matrix}
				0
				&\vline
				&
				\frac16
				&
				-\frac13
				&
				\frac16
				\\
				&\vline
				\\
				\frac12
				&\vline
				&
				\frac16
				&
				\frac5{12}
				&
				-\frac1{12}
				\\
				&\vline
				\\
				1
				&\vline
				&
				\frac16
				&
				\frac23
				&
				\frac16
				\\
				&\vline
				\\
				\hline
				&\vline
				\\
				&
				\vline
				&
				\frac16
				&
				\frac23
				&
				\frac16
			\end{matrix}
		\end{align*}
	\end{subtable}
\end{table}
\begin{table}
	\caption{Coefficients of Runge Kutta schemes of order five (see also ~\cite{HairerLubichWanner2006,HairerNorsettWanner1987,HairerWanner1991}).
	\label{table:irk:5}}
		\begin{subtable}{0.45\textwidth}
		\caption{Coefficients for the Radau~IA method of order five (cf.~\cite[Table IV.5.3.]{HairerWanner1991}).
			\label{table:irk:RadauIA:5}
		}
		\centering
		\begin{align*}
			\begin{matrix}
				0
				&
				\vline
				&
				\frac19
				&
				\frac{-1-\sqrt{6}}{18}
				&
				\frac{-1+\sqrt{6}}{18}
				\\
				&
				\vline
				\\
				\frac{6-\sqrt{6}}{10}
				&
				\vline
				&
				\frac19
				&
				\frac{88+7\sqrt{6}}{360}
				&
				\frac{88-43\sqrt{6}}{360}
				\\
				&
				\vline
				\\
				\frac{6+\sqrt{6}}{10}
				&
				\vline
				&
				\frac19
				&
				\frac{88+43\sqrt{6}}{360}
				&
				\frac{88-7\sqrt{6}}{360}
					\\
				&\vline
				\\
				\hline
				&\vline
				\\
				&\vline
				&
				\frac19
				&
				\frac{16+\sqrt{6}}{36}
				&
				\frac{16-\sqrt{6}}{36}
			\end{matrix}
		\end{align*}
	\end{subtable}
	\begin{subtable}{0.45\textwidth}
		\caption{Coefficients for the Radau~IIA method of order five (cf.~\cite[Table IV.5.5.]{HairerWanner1991}).
			\label{table:irk:RadauIIA:5}
		}
		\centering
		\begin{align*}
			\begin{matrix}
				\frac{4-\sqrt{6}}{10}
				&\vline
				&
				\frac{88-7\sqrt{6}}{360}
				&
				\frac{296-169\sqrt{6}}{1800}
				&
				\frac{-2+3\sqrt{6}}{225}
				\\
				&\vline
				\\
				\frac{4+\sqrt{6}}{10}
				&
				\vline
				&
				\frac{296+169\sqrt{6}}{1800}
				&
				\frac{88+7\sqrt{6}}{360}
				&
				\frac{-2-3\sqrt{6}}{225}
				\\
				&\vline
				\\
				1
				&
				\vline
				&
				\frac{16-\sqrt{6}}{36}
				&
				\frac{16+\sqrt{6}}{36}
				&
				\frac19
				\\
				&\vline
				\\
				\hline
				&\vline
				\\
				&\vline
				&
				\frac{16-\sqrt{6}}{36}
				&
				\frac{16+\sqrt{6}}{36}
				&
				\frac19
			\end{matrix}
		\end{align*}
	\end{subtable}
\end{table}
\begin{table}
	\caption{Coefficients of Runge Kutta schemes of order six(see also ~\cite{HairerLubichWanner2006,HairerNorsettWanner1987,HairerWanner1991}).
	\label{table:irk:6}}

	\begin{subtable}{0.45\textwidth}
		\caption{Coefficients for the Gauss scheme of order six (cf.~\cite[Table IV.5.2.]{HairerWanner1991}).
			\label{table:irk:Gauss:6}
			}
		\centering
		\begin{align*}
			\begin{matrix}
				\frac12 - \frac{\sqrt{15}}{10}
				&
				\vline
				&
				\frac5{36}
				&
				\frac29-\frac{\sqrt{15}}{15}
				&
				\frac5{36}-\frac{\sqrt{15}}{30}
				\\
				&\vline
				\\
				\frac12
				&
				\vline
				&
				\frac5{36}+\frac{\sqrt{15}}{24}
				&
				\frac29
				&
				\frac5{36}
				-
				\frac{\sqrt{15}}{24}
				\\
				&\vline
				\\
				\frac12+\frac{\sqrt{15}}{10}
				&\vline
				&
				\frac5{36}+\frac{\sqrt{15}}{30}
				&
				\frac29+\frac{\sqrt{15}}{15}
				&
				\frac5{36}
				\\
				&\vline
				\\
				\hline
				&\vline
				\\
				&\vline
				&
				\frac5{18}
				&
				\frac49
				&
				\frac5{18}
			\end{matrix}
		\end{align*}
	\end{subtable}
	\begin{subtable}{0.45\textwidth}
		\caption{Coefficients for the Lobatto~IIIC method of order six(cf.~\cite[Table IV.5.11.]{HairerWanner1991}).
			\label{table:irk:LobattoIIIC:6}
		}
		\centering
		\begin{align*}
			\begin{matrix}
				0
				&
				\vline
				&
				\frac1{12}
				&
				\frac{-\sqrt{5}}{12}
				&
				\frac{\sqrt{5}}{12}
				&
				\frac{-1}{12}
				\\
				&\vline
				\\
				\frac{5-\sqrt{5}}{10}
				&
				\vline
				&
				\frac1{12}
				&
				\frac14
				&
				\frac{10-7\sqrt{5}}{60}
				&
				\frac{\sqrt{5}}{60}
				\\
				&\vline
				\\
				\frac{5+\sqrt{5}}{10}
				&
				\vline
				&
				\frac1{12}
				&
				\frac{10+7\sqrt{5}}{60}
				&
				\frac14
				&
				\frac{-\sqrt{5}}{60}
				\\
				&\vline
				\\
				1
				&
				\vline
				&
				\frac1{12}
				&
				\frac5{12}
				&
				\frac5{12}
				&
				\frac1{12}
				\\
				&\vline
				\\
				\hline&\vline
				\\
				&\vline
				&
				\frac1{12}
				&
				\frac5{12}
				&
				\frac5{12}
				&
				\frac1{12}
			\end{matrix}
		\end{align*}
	\end{subtable}

	\begin{subtable}{0.45\textwidth}
		\caption{Coefficients for the Lobatto~IIIA method of order six (cf.~\cite[Table IV.5.7.]{HairerWanner1991}).
			\label{table:irk:LobattoIIIA:6}
		}
		\centering
		\begin{align*}
			\begin{matrix}
				0
				&
				\vline
				&
				0
				&
				0
				&
				0
				&
				0
				\\
				&\vline
				\\
				\frac{5-\sqrt{5}}{10}
				&
				\vline
				&
				\frac{11+\sqrt{5}}{120}
				&
				\frac{25-\sqrt{5}}{120}
				&
				\frac{25-13\sqrt{5}}{120}
				&
				\frac{-1+\sqrt{5}}{120}
				\\
				&\vline
				\\
				\frac{5+\sqrt{5}}{10}
				&
				\vline
				&
				\frac{11-\sqrt{5}}{120}
				&
				\frac{25+13\sqrt{5}}{120}
				&
				\frac{25-+\sqrt{5}}{120}
				&
				\frac{-1-\sqrt{5}}{120}
				\\
				&\vline
				\\
				1
				&
				\vline
				&
				\frac1{12}
				&
				\frac5{12}
				&
				\frac5{12}
				&
				\frac1{12}
				\\
				&\vline
				\\
				\hline
				&
				\vline
				\\
				&\vline
				&
				\frac1{12}
				&
				\frac5{12}
				&
				\frac5{12}
				&
				\frac1{12}
			\end{matrix}
		\end{align*}
	\end{subtable}
	\begin{subtable}{0.45\textwidth}
		\caption{Coefficients for the Lobatto~IIIB method of order six (cf.~\cite[Table IV.5.9.]{HairerWanner1991}).
			\label{table:irk:LobattoIIIB:6}
		}
		\centering
		\begin{align*}
			\begin{matrix}
				0
				&
				\vline
				&
				\frac1{12}
				&
				\frac{-1-\sqrt{5}}{24}
				&
				\frac{-1+\sqrt{5}}{24}
				&
				0
				\\
				&\vline
				\\
				\frac{5-\sqrt{5}}{10}
				&
				\vline
				&
				\frac1{12}
				&
				\frac{25+\sqrt{5}}{120}
				&
				\frac{25-13\sqrt{5}}{120}
				&
				0
				\\
				&\vline
				\\
				\frac{5+\sqrt{5}}{10}
				&
				\vline
				&
				\frac1{12}
				&
				\frac{25+13\sqrt{5}}{120}
				&
				\frac{25-\sqrt{5}}{120}
				&
				0
				\\
				&\vline
				\\
				1
				&
				\vline
				&
				\frac1{12}
				&
				\frac{11-\sqrt{5}}{24}
				&
				\frac{11+\sqrt{5}}{24}
				&
				0
				\\
				&\vline
				\\
				\hline
				&
				\vline
				\\
				&\vline
				&
				\frac1{12}
				&
				\frac5{12}
				&
				\frac5{12}
				&
				\frac1{12}
			\end{matrix}
		\end{align*}
	\end{subtable}
\end{table}

\begin{remark}
	In some cases the adjoint schemes of the Runge-Kutta discretizations are well known schemes of their own:
\begin{itemize}
	\item The scheme for the adjoint discretization of the Gauss scheme is the Gauss scheme itself.
	\item The scheme for the adjoint discretization of the Lobatto~IIIA scheme is the Lobatto~IIIB scheme and vice versa (see also \cite{HairerLubichWanner2006}).
	\item The scheme for the fourth order adjoint discretization of the Lobatto~IIIC scheme is known as Butcher's Lobatto scheme. 
		This scheme is not $A$-stable (see \cite[Example IV.3.5.]{HairerWanner1991}).
	\item The scheme for the adjoint discretization of the Radau~IA scheme is known not to be $A$-stable (see \cite[Example IV.3.5.]{HairerWanner1991}).
\end{itemize}
\end{remark}
Next we investigate the convergence of implicit Runge-Kutta schemes for optimal control problems.

\section{Convergence order of the Runge-Kutta discretizations}
\label{sec:conv}
For the convergence of the Runge-Kutta discretization of the optimal control problem, one could check the order conditions. 
But we want to further classify the schemes, for which the order conditions for optimal control problems hold. 
Therefore we recall the simplifying assumptions on the coefficients of a Runge-Kutta scheme. 
These conditions were introduced for the construction of implicit Runge-Kutta schemes.
\begin{assumption}[Simplifying assumptions] \cite[Chapter IV.5]{HairerWanner1991}
	The simplifying assumptions are given by
	\begin{align}
		\sum_{i=1}^sb_ic_i^{q-1}
		&=
		\frac1q
		,
		&\text{for\ }q&=1,\dots,p
		,
		\tag{B$(p)$}
		\label{eq:simp_assumpt:B_p}
		\\
		\sum_{j=1}^s
		a_{ij}c_j^{q-1}
		&=
		\frac{c_i^q}q
		,
		&\text{for\ }i&=1,\dots,s,\ q=1,\dots,\eta
		,
		\tag{C$(\eta)$}
		\label{eq:simp_assumpt:C_eta}
		\\
		\sum_{i=1}^s
		b_i c_i^{q-1}a_{ij}
		&=
		\frac{b_j}q 
		\left(
			1-c_j^q
		\right)
		&
		\text{for\ }j&=1,\dots,s,\ 
		q=1,\dots,\zeta.
		\tag{D$(\zeta)$}
		\label{eq:simp_assumpt:D_zeta}
	\end{align}
\end{assumption}
Note that the condition \eqref{eq:simp_assumpt:D_zeta} for $\zeta=1$ is equivalent to
\begin{align*}
	d_j = \sum_{i=1}^s b_i a_{ij} = b_j \left(1-c_j\right),
\end{align*}
which will be often used in the proofs later on.
So we can characterize easily the order four schemes, which fulfill the additional order conditions automatically.
\begin{theorem}
	Every third or fourth order Runge-Kutta scheme, for which the simplifying assumption
	\eqref{eq:simp_assumpt:D_zeta} for $\zeta=1$ holds, fulfills the additional order conditions 
	of order three or four respectively.
	\label{theorem:additional_cond:RK:3:4}
\end{theorem}
\begin{proof}
	This proof can be done with the same ideas as the proof of \cite[Proposition 6.1]{Hager2000} for explicit Runge-Kutta schemes.
	With the condition \eqref{eq:simp_assumpt:D_zeta} for $\zeta=1$ the additional conditions of order three and four follow directly of the order conditions from the implicit Runge-Kutta scheme, see \cite[Proposition 6.1]{Hager2000}.
\end{proof}
\begin{corollary}
	The St\"ormer-Verlet scheme applied to an optimal control problem gives a second order approximation, 
	the application of the two stage Radau IA and Radau IIA schemes gives  approximation of order three
	and the application of the two stage Gauss and the three stage Lobatto IIIA, Lobatto IIIB 
	or Lobatto IIIC schemes gives approximations of order four.
\end{corollary}
\begin{proof}
	As the scheme of Tables~\ref{table:irk:SV:2} is only of second order, no further conditions must be fulfilled.
	As seen in \cite[Table IV.5.13]{HairerWanner1991} the simplifying assumptions holds for the discussed collocation methods, 
	so this corollary follows directly of the Theorem~\ref{theorem:additional_cond:RK:3:4}.
\end{proof}
Next we discuss the convergence of the remaining fourth order scheme.
\begin{theorem}
	The pairing
	of the fourth order SDIRK scheme of
	Table~\ref{table:irk:SDIRK:4} with the corresponding adjoint scheme 
	applied to an optimal control problem 
	provides only a second order approximation.
	\label{theorem:SDIRK:reduced_order}
\end{theorem}
\begin{proof}
	It is well known that the SDIRK scheme of Table~\ref{table:irk:SDIRK:4} is a fourth order scheme,
	see \cite[Table IV.6.5]{HairerWanner1991}.
	For the falsification of the additional order conditions of order three we see that
	\begin{align*}
		\sum_{i=1}^s \frac{d_i^2}{b_i} = \frac{18367}{58800}\ne \frac13,
	\end{align*}
	and therefore the application to optimal control problem is only of order two, as for order two no additional order conditions are needed.
\end{proof}
\begin{remark}
	It is easy to check that the schemes of Table \ref{table:irk:SDIRK:4} and the corresponding adjoint scheme are both of order four. 
	Nevertheless the pairing applied to optimal control problems is only of order two, so we see that the conditions in Table~\ref{table:orderconditions:irk:oc} are really additional conditions and are not automatically fulfilled for any implicit Runge-Kutta scheme of the corresponding order for ordinary differential equations.
\end{remark}
\begin{remark}
	The result of Theorem \ref{theorem:SDIRK:reduced_order} is not a general property of SDIRK schemes. 
	There are also SDIRK schemes for which in the discretization \eqref{eq:timediscr:rk}, \eqref{eq:timediscr:connection:schemes} the convergence order is preserved, e.g.~the SDIRK methods denoted to Crouzeix and Raviart in 
	\cite[Exercise IV.6.1]{HairerWanner1991}, \cite[Table II.7.2]{HairerNorsettWanner1987} of order four with three stages and oder three with two stages.
\end{remark}
After the classification of fourth order Runge-Kutta schemes for optimal control, we now consider fifth order schemes. 
\begin{theorem}
	If a Runge-Kutta scheme of order five fulfills the simplifying assumptions
	\eqref{eq:simp_assumpt:B_p},
	\eqref{eq:simp_assumpt:C_eta},
	\eqref{eq:simp_assumpt:D_zeta}
	up to $p=2$,  $\eta=2$, $\zeta=2$, 
	then the additional order conditions are also fulfilled.
	\label{theorem:order5:conditions}
\end{theorem}
\begin{proof}
	The full proof is given in the Appendix \ref{appendix:proof:additional:order:5} 
	and done by algebraic manipulation of the additional order condition with the simplifying assumptions and the
	usual order conditions.
\end{proof}
\begin{corollary}
	The three stage Radau IA and Radau IIA implicit Runge-Kutta schemes applied to an optimal control problem are of order five.
\end{corollary}
\begin{proof}
	As seen in \cite[Table IV.5.13]{HairerWanner1991} the schemes fulfill at least the simplifying assumptions
	\eqref{eq:simp_assumpt:B_p},
	\eqref{eq:simp_assumpt:C_eta},
	\eqref{eq:simp_assumpt:D_zeta}
	up to 
	$p=2$,
	$\eta=2$, 
	$\zeta=2$. 
\end{proof}
\begin{theorem}
	If a Runge-Kutta scheme of order six fulfills the simplifying assumptions
	\eqref{eq:simp_assumpt:B_p},
	\eqref{eq:simp_assumpt:C_eta},
	\eqref{eq:simp_assumpt:D_zeta}
	up to 
	$p=4$, 
	$\eta=2$,
	$\zeta=2$, 
	then the additional order conditions are also fulfilled.
	\label{theorem:order6:conditions}
\end{theorem}
\begin{proof}
	The full proof was carried out by hand by the author by algebraic manipulation of the additional order condition 
	with the simplifying assumptions and the usual order conditions. 
	As this tedious proof gives no higher insights and  is, due to the huge number of order conditions, 
	longer as the proof of Theorem \ref{theorem:order5:conditions}
	the details are omitted.
\end{proof}
\begin{corollary}
	The three stage Gauss and the four stage Lobatto IIIA, Lobatto IIIB and Lobatto IIIC implicit Runge-Kutta schemes applied to an optimal control problem are of order six.
\end{corollary}
\begin{proof}
	As seen in \cite[Table IV.5.13]{HairerWanner1991} the schemes  fulfill at least the simplifying assumptions
	\eqref{eq:simp_assumpt:B_p},
	\eqref{eq:simp_assumpt:C_eta},
	\eqref{eq:simp_assumpt:D_zeta}
	up to $p=4$, $\eta=2$, $\zeta=2$.
\end{proof}
With Theorem \ref{theorem:additional_cond:RK:3:4}, Theorem \ref{theorem:order5:conditions} and Theorem \ref{theorem:order6:conditions} 
we have sufficient conditions if the additional order conditions are fulfilled which are easy to check. 
It is open whether these conditions are also necessary or 
if there exists an implicit Runge-Kutta scheme which fulfills the additional order conditions but not the simplifying assumptions.

\begin{remark}[Full discretization]
	In this section the focus was on the time discretization error.
	The full discretization of a parabolic optimal control problem can be handled with the method of lines as in 
	\cite{ApelFlaig2012}. 
	Then the error can be split 	into
	\begin{align*}
		\norm{\bar y(\cdot,t_i) -\bar y_{hi}}_{L^2(\Omega)}
		+
		\norm{\bar p(\cdot,t_i) - \bar p_{hi}}_{L^2(\Omega)}
		&\lesssim
		\norm{\bar y(\cdot,t_i) - y_{h} (t_i)}_{L^2(\Omega)}
		+
		\norm{y_h(t_i) - \bar y_{hi}}_{L^2(\Omega)}
		\\
		&+
		\norm{\bar p(\cdot,t_i) - p_{h} (t_i)}_{L^2(\Omega)}
		+
		\norm{p_h(t_i) - \bar p_{hi}}_{L^2(\Omega)}
		,
	\end{align*}
	where the functions $y_h$ and $p_h$ are discretized in space with a finite element method.
\end{remark}
\begin{remark}[Regularity]
	The order conditions in this section were taken from \cite{BonnansLaurent-Varin2004,BonnansLaurent-Varin2006,Hager2000}
	and derived with techniques based on Taylor series.
	Therefore high regularity assumptions and smooth solutions are needed to observe these rates.
	For a reduction of the required regularity one might use 
	generalized Taylor polynomials as in the work by Dupont and Scott \cite{DupontScott1980}, 
	this is work of further research.
\end{remark}

%
%
\section{Numerical examples}
\label{sec:Num:Ex}
After the classification of the Runge-Kutta schemes we consider in this section a numerical example which confirms the predicted convergence rates.

As in \cite{ApelFlaig2012,Flaig2013} we solve the discretization \eqref{eq:timediscr:rk} as a system of linear equation for the vector of  unknowns 
\begin{align*}
	\left(\bar y_{h1},\dots,\bar  y_{hN},\bar p_{h0},\dots,\bar p_{hN},\bar y_{h0;1},\dots,\bar y_{hN;s},\bar p_{h0;1},\dots,\bar p_{hN;s} \right)^T.
\end{align*}

For the numerical examples we consider the optimal control problem
\begin{align}
	\left.
	\begin{aligned}
	\min 
	\frac12 \norm{ y(\cdot,T) -y_D}_{L^2(\Omega)}^2
	&+
	\frac\nu2 \int_0^T \norm{u}_{L^2(\Omega)}^2 \dd t
	,
	\\
	y_t- \Laplace y &= u
	,
	&\text{in }\Omega&\times (0,T]
	,
	\\
	\frac{\partial y}{\partial n} &=0
	,
	&\text{on }\partial\Omega&\times (0,T]
	,
	\\
	y(\cdot,0)&=v
	,
	&\text{in }\Omega
	,
	\end{aligned}
	\right\}
	\label{eq:numex:example:terminal}
\end{align}
with $\Omega=(0,1)$, $T=1$ and $y_D=v = \sqrt{2} \cos(\pi x)$.
\begin{remark}(See also \cite{Flaig2013}.)
	The analytic solution of the optimal control problem \eqref{eq:intro:parabol:ocp} 
	with $B=M=I$ and a self-adjoint elliptic operator $A$ 
	can be given as eigenfunction series~(see~\cite{HouImanuvilovKwon2006}).
	Let $\{e_i\}_{i=0}^\infty$ and $\{\lambda_i\}_{i=0}^\infty$ be the series of eigenfunctions and eigenvalues of the spatial operator~$A$. 
	If the data are given as eigenfunction expansions
        \begin{align}
                v &= \sum_{k=0}^\infty y_{0,k} e_i,
                &
                y_D &= \sum_{k=0}^\infty y_{D,k} e_i.
                \label{eq:discr:cnsv:egfct:data}
        \end{align}
        The optimal control problem decouples into independent problems for every eigenfunction $e_i$ and has the solution
        \begin{align}
                \bar y&= \sum_{i=0}^\infty C_{1,i} e_i \e^{ \lambda_i t} +C_{2,i} e_i \e^{- \lambda_i t}, 
                &
                \bar p&= \sum_{i=0}^\infty C_{3,i} e_i \e^{ \lambda_i t} .\label{eq:discr:cnsv:egfct:sol}
        \end{align}
	The coefficients can be computed with Maple and are given in Table \ref{tab:data_ana}. 

	For the example \eqref{eq:numex:example:terminal} with $y_D=v = \sqrt{2} \cos(\pi x)$  the series for the state and the adjoint state  reduce to the terms with the second eigenfunction $e_1 =   \sqrt{2} \cos(\pi x)$ of the Laplace operator with Neumann boundary conditions,
	i.e.~only the coefficients $C_{1,1}$, $C_{2,1}$ and $C_{3,1}$ do not vanish.
        \begin{table}
                \caption{Coefficients for the exact solution (\ref{eq:discr:cnsv:egfct:sol}) of the problem \eqref{eq:numex:example:terminal} to the data (\ref{eq:discr:cnsv:egfct:data}).
                \label{tab:data_ana}
                }
                \begin{center}
                        \begin{tabular}{ccccc}
                        \toprule
                        $y_{0,i}$       &       $y_{D,i}$       &       $C_{1,i}$       &       $C_{2,i}$       &       $C_{3,i}$
                        \tabularnewline
                        \midrule        
                        \addlinespace
                        $a_i$   
                        &       $b_i$
                        &       $\frac{-b_i+a_i \e^{-\lambda_i}}{-2 \nu \lambda_i \e^{\lambda_i} -\e^{\lambda_i}+\e^{-\lambda_i}}$
                        &       $-\frac{-b_i+a_i\e^{\lambda_i} +2\nu \lambda_i a_i  \e^{\lambda_i} }{-2\nu\lambda_i \e^{\lambda_i}   - \e^{\lambda_i} +\e^{-\lambda_i}} $       
                        &       $2\lambda_i\nu C_{1,i}$ 
                        \tabularnewline
                        \bottomrule
                        \end{tabular}
                \end{center}
        \end{table}
	\label{remark:represofsol}
\end{remark}
The spatial discretization is adapted to the time discretization. 
The polynomial degree of the Lagrange finite elements for the spatial discretization is chosen as $k-1$ for time discretization schemes of order $k$. 
So an error splitting argument provides the error bound
\begin{align*}
	\norm{\bar y(\cdot,t_i) -\bar y_{hi}}_{L^2(\Omega)}
	+
	\norm{\bar p(\cdot,t_i) -\bar p_{hi}}_{L^2(\Omega)}
	\lesssim
	h^k + \tau^k.
\end{align*}
In the numerical examples the discretization parameters $\tau$ and $h$ are chosen so that $\tau\sim h$.

We measure the time discretization error by the quantities 
\begin{align}
	\max_{i\in \{0,1,\cdots ,N\}}
	\left(
		(\bar y_{hi} - I_h \bar y(x,t_i))^T
		M 
		(\bar y_{hi} - I_h \bar y(x,t_i))
	\right)^{\frac12},
	\label{eq:error:state}
	\\
	\max_{i\in \{0,1,\cdots ,N\}} 
	\left(
		(\bar p_{hi} - I_h \bar p(x,t_i))^T
		M 
		(\bar p_{hi} - I_h \bar p(x,t_i))
	\right)^{\frac12},
	\label{eq:error:adjointstate}
\end{align}
where $I_h $ is the Lagrangian interpolation operator to the corresponding spatial discretization  and $M$ the finite element mass matrix.
In Figure \ref{fig:ocp:terminal:i} to Figure \ref{fig:ocp:terminal:schemes:order:6} 
\begin{figure}
	\begin{subfigure}{0.5\textwidth}
		\centering
		\input{scheme0-tex}
		\caption{Second order convergence of the discretization based on the St\"ormer-Verlet discretization of Table~\ref{table:irk:SV:2}.
		\label{fig:ocp:terminal:SV}
		}
	\end{subfigure}
	\begin{subfigure}{0.5\textwidth}
		\centering
		\input{scheme8-tex}
		\caption{Second order convergence of the discretization based on the SDIRK  scheme of Table~\ref{table:irk:SDIRK:4}.
		\label{fig:ocp:terminal:SDIRK:4}
		}
	\end{subfigure}
		\caption{Observed convergence order two of the numerical approximation of the example \eqref{eq:numex:example:terminal}.
	\label{fig:ocp:terminal:i}
	}
\end{figure}
\begin{figure}
	\begin{subfigure}{0.5\textwidth}
	\centering
		\input{scheme2-tex}
		\caption{Third order convergence of the discretization based on the Radau IA scheme of Table~\ref{table:irk:RadauIA:3}.}
	\end{subfigure}
	\begin{subfigure}{0.5\textwidth}
	\centering
		\input{scheme3-tex}
		\caption{Third order convergence of the discretization based on the Radau IIA scheme of Table~\ref{table:irk:RadauIIA:3}.}
	\end{subfigure}
	\caption{Observed convergence order three of the numerical approximation of the example~\eqref{eq:numex:example:terminal}.
	\label{fig:ocp:terminal:i-i}
	}
\end{figure}
\begin{figure}
	\begin{subfigure}{0.5\textwidth}
	\centering
		\input{scheme1-tex}
		\caption{Fourth order convergence of the discretization based on the  Gauss scheme of Table~\ref{table:irk:Gauss:4}.}
	\end{subfigure}
	\begin{subfigure}{0.5\textwidth}
	\centering
		\input{scheme4-tex}
		\caption{Fourth order convergence of the discretization based on the Lobatto IIIA scheme of Table~\ref{table:irk:LobattoIIIA:4}.}
	\end{subfigure}

	\begin{subfigure}{0.5\textwidth}
	\centering
		\input{scheme5-tex}
		\caption{Fourth order convergence of the discretization based on the Lobatto IIIB scheme of Table~\ref{table:irk:LobattoIIIB:4}.}
	\end{subfigure}
	\begin{subfigure}{0.5\textwidth}
	\centering
		\input{scheme6-tex}
		\caption{Forth order convergence of the discretization based on the Lobatto IIIC scheme  of Table~\ref{table:irk:LobattoIIIC:4}.}
	\end{subfigure}
	\caption{Observed convergence order four of the numerical approximation of the example~\eqref{eq:numex:example:terminal}.
	\label{fig:ocp:terminal:ii}
	}
\end{figure}
we observe nicely the predicted convergence rates for the example \eqref{eq:numex:example:terminal} with $\nu=0.001$. 
In the computations with some fourth and sixth order schemes we also observe the influence of the round-off error due to the high numbers of unknowns. All the computations were done in Matlab. 

The predicted order reduction for the SDIRK method can be seen in Figure \ref{fig:ocp:terminal:SDIRK:4}. 
For spatial discretization of the numerical example with the SDIRK time discretization cubic Lagrange finite elements are used, 
as for the other fourth order time discretization schemes.
\begin{figure}
	\begin{subfigure}{0.5\textwidth}
	\centering
		\input{scheme12-tex}
		\caption{Fifth order convergence of the discretization based on the three stage Radau IA scheme.}
	\end{subfigure}
	\begin{subfigure}{0.5\textwidth}
	\centering
		\input{scheme13-tex}
		\caption{Fifth order convergence of the discretization based on the three stage Radau IIA scheme.}
	\end{subfigure}
	\caption{Observed convergence order five of the numerical approximation of the example \eqref{eq:numex:example:terminal}.
	\label{fig:ocp:terminal:schemes:order:5:6}
	}
\end{figure}
\begin{figure}
	\begin{subfigure}{0.5\textwidth}
	\centering
		\input{scheme11-tex}
		\caption{Sixth order convergence of the discretization based on the three stage Gauss scheme.}
	\end{subfigure}
	\begin{subfigure}{0.5\textwidth}
	\centering
		\input{scheme14-tex}
		\caption{Sixth order convergence of the discretization based on the four stage Lobatto IIIA scheme.}
	\end{subfigure}
	\begin{subfigure}{0.5\textwidth}
	\centering
		\input{scheme15-tex}
		\caption{Sixth order convergence of the discretization based on the four stage Lobatto IIIB scheme.}
	\end{subfigure}
	\begin{subfigure}{0.5\textwidth}
	\centering
		\input{scheme16-tex}
		\caption{Sixth order convergence of the discretization based on the four stage Lobatto IIIC scheme.}
	\end{subfigure}
	\caption{Observed convergence order six of the numerical approximation of the example \eqref{eq:numex:example:terminal}.
	\label{fig:ocp:terminal:schemes:order:6}
	}
\end{figure}
\begin{remark}
	The order reduction of the SDIRK method can also be observed for an optimal control problem 
	with one linear ordinary differential equation.
	Consider the optimal control problem
	\begin{align}
	\left.
	\begin{aligned}
	\min 
	\frac12 (y(1) -1)^2
	&+
	\frac\nu2 \int_0^1 u^2 \dd t
	,
	\\
	y_t +\pi^2 y &= u
	,
	&\text{for } t &\in (0,1]
	,
	\\
	y(0)&=1
	.
	\end{aligned}
	\right\}
	\label{eq:oc:ode}
	\end{align}
	Even for this very simple example we observe the reduced convergence rate in Figure \ref{fig:SDIRK:ode}.
	Again the regularization parameter $\nu=0.001$ was chosen.
	\begin{figure}
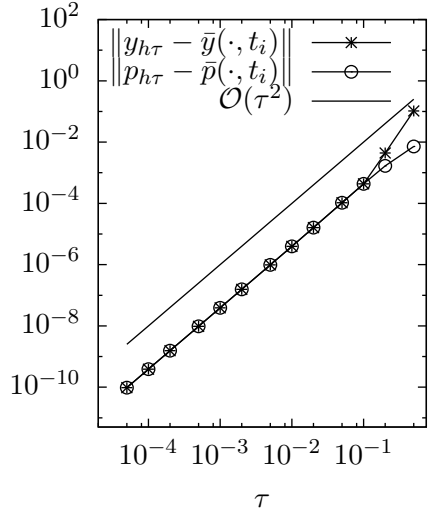

		\centering
		\include{scheme8-spectral-tex}
		\caption{Observed convergence for example \eqref{eq:oc:ode} with the SDIRK method.
		\label{fig:SDIRK:ode}
		}
	\end{figure}
\end{remark}
	
\begin{remark}
	The optimal control problem \eqref{eq:oc:ode} can be interpreted as a spatial Galerkin discretization of optimal control Problem \eqref{eq:numex:example:terminal}, where the bases of trial and test space  are chosen as the second normalized eigenfunction of the Laplace operator. Note that the first eigenfunction of the Laplace operator with Neumann boundary conditions is the constant function. 
\end{remark}
\begin{remark}
	In Figure \ref{fig:ocp:terminal:SV} we observe the second order convergence of the St\"ormer-Verlet scheme.
	Similar observations were presented in \cite{Flaig2013}.
	But in contrast to \cite{Flaig2013}, where the convergence of the state was observed
	in the time discretization points $t_i$ and the convergence of the adjoint state was observed 
	in the time middle points $t_{i+\frac12}=\frac{t_i+t_{i+1}}2$, 
	we present in Figure \ref{fig:ocp:terminal:SV} the convergence 
	of the state and the adjoint state in the time discretization points $t_i$.
\end{remark}

%
%
\section{Conclusions and Outlook}
In this paper we discussed the use of higher order implicit Runge-Kutta schemes for optimal control with parabolic partial differential equations for which optimization and discretization commute.  
In terms of the well known simplifying assumptions on the coefficients of  implicit Runge-Kutta scheme we were able to give
a classification for which discretization schemes up to order six the convergence order is preserved. 
For collocation schemes of Gauss, Radau IA, Radau IIA, Lobatto IIIA, Lobatto IIIB and Lobatto IIIC type and a SDIRK scheme the expected and the numerical convergence rates coincide nicely.

For schemes of order higher than six the order conditions are not known explicitly, but they can be computed with the aid of bi-colored Butcher trees, as described in \cite{BonnansLaurent-Varin2004,BonnansLaurent-Varin2006}. 
For a reduction of the additional order conditions of order higher as six the procedure presented in this paper is not practical due the huge number of additional conditions. Therefore a more elegant technique should be developed for the classification of schemes of order higher than six.

%
%
\section*{Acknowledgements}

The work was partially supported by the DFG priority program SPP 1253.

\appendix
\section{Proof of Theorem \ref{theorem:order5:conditions}}
\label{appendix:proof:additional:order:5}
\begin{proof}[Full proof of Theorem \ref{theorem:order5:conditions}]
	The idea of the proof is to use the simplifying assumptions 
	\eqref{eq:simp_assumpt:B_p},
	\eqref{eq:simp_assumpt:C_eta}, 
	\eqref{eq:simp_assumpt:D_zeta}
	to reduce the additional order conditions to the classic order conditions or 
	order conditions of lower order, which have already been reduced to the order conditions of the uncontrolled system.
	As all the numerical schemes fulfill the order conditions for the uncontrolled systems, 
	these conditions can be used to calculate the value of the reduced expression.

	Surely the way of the application of the simplifying assumptions is not unique, here one possibility is presented.
	A first goal in the reduction of order conditions with a fraction $\frac\cdot{b_i}$
	is to use \eqref{eq:simp_assumpt:D_zeta} to produce an additional $b_i$ which cancels out.
	In the following we discuss the reduction of all the additional order conditions.
	\begin{enumerate}
		\item 
			For the first additional order condition of \eqref{eq:A-O-5-1} we use the simplifying assumption \eqref{eq:simp_assumpt:D_zeta} for $\zeta=1$, the last condition of \eqref{eq:O-4} and the first condition of  \eqref{eq:O-5-3}. 
			This yields
			\begin{align*}
				\sum_{kl} \frac1{b_k}a_{lk}c_kd_kd_l
				&=
				\sum_{kl} d_la_{lk}c_k - \sum_{kl}a_{lk}c_k^2d_l
				=\frac1{24}-\frac1{60}=\frac1{40}.
			\end{align*}
		\item 
			For the second additional order condition of \eqref{eq:A-O-5-1} we use
			the simplifying assumption \eqref{eq:simp_assumpt:D_zeta} for $\zeta=1$ and the third condition of
			\eqref{eq:O-4} and the second condition of \eqref{eq:O-5-3} to get
			\begin{align*}
				\sum_k \frac1{b_k} c_k^2 d_k^2
				&=
				\sum_k c_k^2 d_k
				-
				\sum_k c_k^3 d_k
				=\frac1{12}-\frac1{20}=\frac1{30}
				.
			\end{align*}
		\item 
			For the last order condition of \eqref{eq:A-O-5-1} we use again
			the simplifying assumption \eqref{eq:simp_assumpt:D_zeta} for $\zeta=1$, the first condition of \eqref{eq:O-3} the third condition of \eqref{eq:O-4} and the first condition of \eqref{eq:O-5-2}. This gives
			\begin{align*}
				\sum_l \frac1{b_l^2}c_ld_l^3
				&=
				\sum c_l d_l 
				-2 \sum c_l^2d_l
				+\sum_lc_l^3d_l = \frac16 - \frac2{12}-\frac1{20}=\frac1{20}
				.
			\end{align*}
		\item 
			For the first condition of \eqref{eq:A-O-5-2} we apply
			the simplifying assumption \eqref{eq:simp_assumpt:D_zeta} for $\zeta=1$ 
			and use the last condition of \eqref{eq:O-4} and the second condition of \eqref{eq:O-5-1}, which gives
			\begin{align*}
				\sum_{kl} \frac{1}{b_k}a_{kl}c_l d_k^2
				&=
				\sum_{kl} a_{kl}c_l d_k 
				-
				\sum_{kl} a_{kl}c_l d_k c_k
				=
				\frac1{24}-\frac1{40}=\frac1{60}
				.
			\end{align*}
		\item 
			For the second condition of \eqref{eq:A-O-5-2} we use
			the simplifying assumption \eqref{eq:simp_assumpt:D_zeta} for $\zeta=1$, 
			the condition \eqref{eq:O-2},
			the first condition of \eqref{eq:O-3},
			the third condition of \eqref{eq:O-4}
			and
			the first condition of \eqref{eq:O-5-2}
			to end with
			\begin{align*}
				\sum_m \frac1{b_m^3}d_m^4
				&=
				\sum_m 
				d_m 
				-
				\sum_m 
				3d_mc_m
				+
				\sum_m 
				3 d_m c_m^2
				-
				\sum_m 
				d_m c_m^3
				=
				\frac15
				.
			\end{align*}
		\item 
			For the third condition of \eqref{eq:A-O-5-2} we apply
			the simplifying assumption \eqref{eq:simp_assumpt:C_eta} for $\eta=2$ twice 
			and get with the second condition of \eqref{eq:O-5-3} the result
			\begin{align*}
				\sum_{ijk}b_ia_{ik}a_{ij}c_jc_k
				&=
				\sum_i b_i \left(\sum_k a_{ik}c_k\right) \left(\sum_j a_{ij}c_j\right)
				=
				\frac14 \sum_i b_i c_i^4 = \frac1{20}
				.
			\end{align*}
		\item 
			For the first condition of \eqref{eq:A-O-5-3} we apply again 
			the simplifying assumption \eqref{eq:simp_assumpt:C_eta} for $\eta=2$ 
			and the use of the first condition of \eqref{eq:O-5-3} yields
			\begin{align*}
				\sum_{jkl} a_{lk}a_{kj}c_j d_l
				&=
				\sum_{kl} a_{lk}d_l \left(\sum_j a_{kj} c_j\right)
				=
				\frac12
				\sum_{kl} a_{lk}d_l c_k^2
				=
				\frac1{120}
				.
			\end{align*}
		\item 
			For the second condition of \eqref{eq:A-O-5-3} we apply first
			the simplifying assumptions  \eqref{eq:simp_assumpt:D_zeta} for $\eta=1$ 
			and 
			then the definition of $c_l$ and 
			the simplifying assumption \eqref{eq:simp_assumpt:C_eta} for $\eta=2$.
			Together with the third condition of \eqref{eq:O-4} and the first condition of \eqref{eq:O-5-2}
			this gives
			\begin{align*}
				\sum_{kl}\frac1{b_k} a_{lk}d_kc_ld_l
				&=
				\sum_l c_l d_l \left(\sum_k  a_{lk}\right)
				-
				\sum_l c_l d_l \left(\sum_k  a_{lk} c_k\right)
				=
				\sum_l c_l^2d_l - \frac12 \sum_l c_l^3d_l 
				=
				\frac7{120}
				.
			\end{align*}
		\item 
			For the last condition of \eqref{eq:A-O-5-3} we apply the simplifying assumption 
			\eqref{eq:simp_assumpt:D_zeta} for $\eta=2$ twice 
			and get with
			\eqref{eq:O-1}, 
			the second condition of \eqref{eq:O-2}
			and the second condition of \eqref{eq:O-5-3} the result
			\begin{align*}
				\sum_{ijk} \frac{b_ib_j}{b_k} a_{jk}a_{ik}c_ic_j
				&=
				\sum_{jk} \frac{b_j}{b_k} a_{jk} c_j \left(\sum_i b_i a_{ik}c_i\right)
				=
				\frac12 \sum_k (1-c_k) \left(\sum_jb_ja_{jk}c_j\right)
				\\
				&=
				\frac14 \sum_k b_k (1-c_k)(1-c_k^2)
				=
				\frac14
				\sum_k\left( b_k - 2 b_kc_k^2 + b_kc_k^4\right)
				=\frac2{15}
				.
			\end{align*}
		\item 
			For the first condition of \eqref{eq:A-O-5-4} we use 
			the simplifying assumption  \eqref{eq:simp_assumpt:D_zeta} for $\eta=1$ and $\eta=2$,
			the second condition of \eqref{eq:O-4},
			the second condition of \eqref{eq:O-3} 
			and the second condition of \eqref{eq:O-5-3} to get
			\begin{align*}
				\sum_{ik} \frac{b_i}{b_k} a_{ik}c_ic_kd_k
				&=
				\sum_{ik} b_i a_{ik}c_ic_k 
				-
				\sum c_k^2 \left( \sum_i b_i c_i a_{ik} \right)
				=
				\frac18
				-
				\frac12 \sum c_k^2 b_k \left(1-c_k^2 \right)
				= \frac7{120}
				.
			\end{align*}
		\item 
			For the second condition of \eqref{eq:A-O-5-4} we use again
			the simplifying assumptions  \eqref{eq:simp_assumpt:D_zeta} for $\eta=1$ and $\eta=2$. 
			The remaining expressions are treated with \eqref{eq:O-1}, 
			the simplifying condition \eqref{eq:simp_assumpt:B_p} for $p=2$,
			the second condition of \eqref{eq:O-3},
			the first condition of \eqref{eq:O-4}
			and the second condition of \eqref{eq:O-5-3}.
			This gives
			\begin{align*}
				\sum_{il} \frac{b_i}{b_l^2}a_{il}c_id_l^2
				&=
				\sum_{il} b_i a_{il} c_i (1-c_l)^2
				=
				\sum_l (1-c_l)^2 \left( \sum_i b_i c_i a_{il} \right)
				=
				\frac12 \sum_l b_l (1-c_l)^2(1-c_l^2)
				\\
				&=
				\frac12 \sum_l \left( b_l -2b_lc_l+2b_lc_l^3-b_lc_l^4 \right)
				=\frac3{20}
				.
			\end{align*}
		\item 
			For the last condition of  \eqref{eq:A-O-5-4} we use first
			the simplifying assumptions  \eqref{eq:simp_assumpt:D_zeta} for $\eta=1$ we get 
			due to symmetry properties
			\begin{align}
				\sum_{lmk} \frac1{b_k}a_{mk}a_{lk}d_l d_m
				&=
				\sum_{lmk} \frac{b_lb_m}{b_k}a_{mk}a_{lk} (1-c_l)(1-c_m)
				\nonumber
				\\
				&=
				\sum_{lmk} \frac{b_lb_m}{b_k}a_{mk}a_{lk} 
				-2 \sum_{lmk} \frac{b_lb_m}{b_k}a_{mk}a_{lk} c_l
				+ \sum_{lmk} \frac{b_lb_m}{b_k}a_{mk}a_{lk} c_lc_m.
				\label{eq:in_proof_reduction_order_5:condition_10}
			\end{align}
			The last term is the third condition of \eqref{eq:A-O-5-3} and therefore we already know how to tread this term.
			On the first term of \eqref{eq:in_proof_reduction_order_5:condition_10} we apply 
			the simplifying assumptions \eqref{eq:simp_assumpt:D_zeta} for $\eta=1$ twice  
			and get with
			\eqref{eq:O-1},
			\eqref{eq:simp_assumpt:B_p} for $p=2$  
			and the second condition of \eqref{eq:O-3}
			\begin{align*}
				\sum_{lmk} \frac{b_lb_m}{b_k}a_{mk}a_{lk} &=
				\sum_k \frac1{b_k} 
				\left(\sum_m b_m a_{mk}\right)
				\left(\sum_l b_l a_{lk}\right)
				=
				\sum_k b_k(1-c_k)^2
				=\frac13.
			\end{align*}
			For the remaining term of \eqref{eq:in_proof_reduction_order_5:condition_10} 
			the use of \eqref{eq:simp_assumpt:D_zeta} for $\eta=1$ and $\eta=2$  
			and
			\eqref{eq:O-1},
			\eqref{eq:simp_assumpt:B_p} for $p=2$,
			the second condition of \eqref{eq:O-2}
			and the first condition of \eqref{eq:O-2}
			yields
			\begin{align*}
				\sum_{lmk} \frac{b_lb_m}{b_k}a_{mk}a_{lk} c_l
				&=
				\sum_k \frac1{b_k} \left( \sum_mb_ma_{mk}\right)\left(\sum_lb_la_{lk}c_l\right)
				=
				\frac12 \sum_k b_k (1-c_k)(1-c_k^2) = \frac5{24}.
			\end{align*}
			Altogether we have
			\begin{align*}
				\sum_{lmk} \frac1{b_k}a_{mk}a_{lk}d_l d_m
				&=
				\frac13-\frac5{12}+\frac2{15}=\frac1{20}.
			\end{align*}
		\item 
			For the first condition of \eqref{eq:A-O-5-5} 
			we start with the use of the simplifying assumption \eqref{eq:simp_assumpt:D_zeta} for $\eta=1$
			and the definition of $c_m$.
			The last condition of \eqref{eq:O-4}, 
			the first condition of \eqref{eq:O-5-3} 
			and the first condition of \eqref{eq:O-3} give
			\begin{align*}
				\sum_{lm} \frac1{b_l^2}a_{ml}d_l^2d_m
				&=
				\sum_{lm} a_{ml} (1-c_l)^2 d_m 
				=
				\sum_{lm} a_{ml}d_m 
				-2 \sum_{lm} a_{ml}c_ld_m
				+
				\sum_{lm}a_{ml}c_l^2d_m
				\\
				&=
				\sum_{m} c_m d_m
				-\frac1{12}
				+\frac1{60}
				=\frac16-\frac1{15}=\frac1{10}
				.
			\end{align*}
		\item 
			For the second condition of \eqref{eq:A-O-5-5}
			the use of \eqref{eq:simp_assumpt:D_zeta} for $\zeta=1$,
			the definition of $c_i$,
			the last condition of \eqref{eq:O-4},
			the simplifying assumption \eqref{eq:simp_assumpt:C_eta} for $\eta=2$ 
			and the first condition of \eqref{eq:O-5-3}
			yields
			\begin{align*}
				\sum_{kml}\frac1{b_k}a_{ml}a_{lk}d_kd_m
				&= \sum_{kml}a_{ml}a_{lk} d_m - \sum_{kml}a_{ml}a_{lk} d_mc_k
				\\
				&= \sum_{ml}a_{ml}c_{l} d_m - \sum_{ml}a_{ml} d_m\left(\sum_ka_{lk}c_k \right)
				=
				\frac1{24}
				-
				\frac12 \sum_{ml} a_{ml}d_mc_l^2
				\\
				&=\frac1{24}-\frac1{120}
				=\frac1{30}
				.
			\end{align*}
		\item 
			For the last condition of \eqref{eq:A-O-5-5}
			we apply the simplifying assumption \eqref{eq:simp_assumpt:D_zeta} for $\zeta=2$,
			the definition of $c_l$,
			first condition of \eqref{eq:O-5-3}
			and the first condition of \eqref{eq:O-3}
			to get
			\begin{align*}
				\sum_{ilk} \frac{b_i}{b_k}a_{lk} a_{ik}c_id_l
				&=
				\sum_{lk} \frac1{b_k} a_{lk}d_l \left(\sum_i b_i a_{ik}c_i\right)
				=
				\frac12 \sum_{lk} a_{lk}d_l - \frac12 \sum_{lk} a_{lk}d_lc_k^2
				\\
				&=
				\frac12 \sum_{l} c_{l}d_l
				-
				\frac1{120}
				=
				\frac3{40}.
			\end{align*}
		\item 
			For the first condition of \eqref{eq:A-O-5-6}  we use 
			the simplifying assumption \eqref{eq:simp_assumpt:D_zeta} for $\zeta=1$ and
			the simplifying assumption \eqref{eq:simp_assumpt:C_eta} for $\zeta=2$ three times.
			With the definition of $c_i$, the first condition of \eqref{eq:O-4} 
			and the second condition of \eqref{eq:O-5-3}
			we get
			\begin{align*}
				\sum_{ikl} \frac{b_i}{b_k} a_{ik}a_{il}d_k c_l
				&=
				\sum_{ik} b_i a_{ik} \left(\sum_l a_{il}c_l\right)
				-
				\sum_i b_i 
				\left(\sum_k a_{ik} c_k\right)
				\left(\sum_l a_{il} c_l\right)
				\\
				&= 
				\frac12 \sum_i b_ic_i^2\left(\sum_k a_{ik}\right)
				-
				\frac14 \sum_i b_i c_i^4
				=
				\frac12\sum_i b_ic_i^3 -\frac1{20}
				=\frac18-\frac1{20}=\frac3{40}.
			\end{align*}
		\item 
			To the second condition of \eqref{eq:A-O-5-6} we apply
			the simplifying assumption \eqref{eq:simp_assumpt:D_zeta} for $\zeta=1$ once, 
			use the definition of $c_i$,
			the third condition of \eqref{eq:A-O-4} 
			and the first condition of \eqref{eq:A-O-5-6}. This yields
			\begin{align*}
				\sum_{ilm}\frac{b_i}{b_lb_m} a_{im}a_{il}d_ld_m
				&=
				\sum_{ilm} \frac{b_i}{b_m} a_{im}a_{il} d_m
				-
				\sum_{ilm} \frac{b_i}{b_m} a_{im}a_{il} c_ld_m
				\\
				&=
				\sum_{ilm} \frac{b_i}{b_m} a_{im}c_{l} d_m
				-
				\sum_{ilm} \frac{b_i}{b_m} a_{im}a_{il} c_ld_m
				=
				\frac2{15}
				.
			\end{align*}
		\item 
			The application of the simplifying assumption \eqref{eq:simp_assumpt:D_zeta} for $\zeta=1$
			to the last condition of \eqref{eq:A-O-5-6} 
			together with 
			the last condition of \eqref{eq:O-5-2}, the definition of $c_i$ and the first condition of \eqref{eq:O-4}
			gives
			\begin{align*}
				\sum_{ik} \frac{b_i}{b_k}a_{ik}c_i^2d_k
				&=
				\sum_{ik}b_ia_{ik}c_i^2
				-
				\sum_{ik}b_ia_{ik}c_i^2c_k
				=
				\sum_{ik}b_ic_i^3
				-
				\frac1{10}
				=\frac3{20}
				.
			\end{align*}
		\item 
			For the condition \eqref{eq:A-O-5-7} we use 
			the simplifying assumption \eqref{eq:simp_assumpt:D_zeta} for $\zeta=1$ two times,
			the definition of $c_l$, 
			the first condition of \eqref{eq:O-3},
			the third and fourth conditions of \eqref{eq:O-4}
			and the second condition of \eqref{eq:O-5-1} and get
			\begin{align*}
				\sum_{lm}\frac1{b_lb_m} a_{lm}d_l^2d_m &=
				\sum_{lm} a_{lm}d_l 
				-
				\sum_{lm} a_{lm}d_l c_l
				-
				\sum_{lm} a_{lm}d_l c_m
				+
				\sum_{lm} a_{lm}d_l c_lc_m
				\\
				&=
				\sum_{lm} c_{l}d_l 
				-
				\sum_{lm} c_{l}^2d_l 
				-
				\sum_{lm} a_{lm}d_l c_m
				+
				\sum_{lm} a_{lm}d_l c_lc_m
				=\frac1{15}
			\end{align*}
	\end{enumerate}
	Altogether we have derived the additional order conditions with the use of the simplifying assumptions
	and \eqref{eq:simp_assumpt:B_p} for $p\leq2$,
	\eqref{eq:simp_assumpt:C_eta} for  $\eta=1,2$,
	\eqref{eq:simp_assumpt:D_zeta} for  $\zeta=1,2$,
	and the order conditions \eqref{eq:O-1}--\eqref{eq:O-4} and \eqref{eq:O-5-1}--\eqref{eq:O-5-3}.	
\end{proof}

%
%

\bibliographystyle{plain}
\bibliography{submission}

\end{document}